\documentclass{elsarticle}
\usepackage{geometry, comment}   
\usepackage[parfill]{parskip}    
\usepackage{graphicx,psfrag}
\usepackage{amssymb,amsopn,amsthm,amsmath}
\usepackage{hyperref}
\usepackage{pdfsync}
\usepackage{footmisc}
\usepackage[all]{xy}
\usepackage{hyperref}
\usepackage[usenames]{color}
\usepackage{caption}
\usepackage{subcaption}

\newtheorem{thm}{Theorem}[section]
\newtheorem{prop}[thm]{Proposition}
\newtheorem{lem}[thm]{Lemma}
\newtheorem{cor}[thm]{Corollary}
\newtheorem*{thm*}{Theorem}

\theoremstyle{definition}

\newtheorem{conj}{Conjecture}
\newtheorem{question}{Question}

\newtheorem{rem}[thm]{Remark}

\newtheorem*{defn*}{Definition}
\newtheorem*{rem*}{Remark}
\newcounter{nmdthmcnt}

\newcommand{\G}{{\Gamma}}
\renewcommand{\a}{{\alpha}}
\renewcommand{\L}{{\Lambda}}

\DeclareMathOperator{\az}{alph}

\DeclareMathOperator{\genus}{genus}

\newcommand{\factor}[2]{{\raise0.7ex\hbox{$#1$} \!\mathord{\left/ {\vphantom {#1 {#2}}}\right.\kern-\nulldelimiterspace}\!\lower0.7ex\hbox{${#2}$}}}

\newcommand{\GG}{\ensuremath{\mathbb{G}}}

\newcommand{\FF}{\ensuremath{\mathbb{F}}}

\newcommand{\NN}{\ensuremath{\mathbb{N}}}

\newcommand{\maps}{\rightarrow}

\begin{document}
\begin{frontmatter}
\title{Embedddings between partially commutative groups: 
two counterexamples\tnoteref{t1}}
\author[mio]{Montserrat Casals-Ruiz\fnref{fn1,fn2}
}
\ead{montsecasals@gmail.com}

\author[ncl]{Andrew Duncan\fnref{fn3}}
\ead{andrew.duncan@ncl.ac.uk}

\author[mio]{Ilya Kazachkov\fnref{fn2,fn4}}
\ead{ilya.kazachkov@gmail.com}

\address[mio]{Mathematical Institute, University of Oxford, 24-29 St. Giles', Oxford, OX1 3LB, UK}
\address[ncl]{School of Mathematics and Statistics, Newcastle University, Newcastle upon Tyne, NE1 7RU, United Kingdom}
\tnotetext[t1]{The authors are grateful to the referee for careful reading of,  and 
several improvements to, the text.} 

\fntext[fn1]{Supported by a Marie Curie International Incoming Fellowship within the 7th European Community Framework Programme.}

\fntext[fn2]{Supported by the Spanish Government, grant MTM2011-28229-C02-02, partly with FEDER funds.}

\fntext[fn3]{Partially supported by EPSRC grant EP/F014945.}

\fntext[fn4]{Supported by the Royal Commission's 1851 Research Fellowship.}

\begin{abstract}
In this note we give two examples of partially commutative subgroups of partially commutative groups. Our examples are counterexamples to the Extension Graph Conjecture and to the Weakly Chordal Conjecture of Kim and Koberda, \cite{KK}. On the other hand we extend the class of partially commutative
groups for which it is known that the Extension Graph Conjecture holds, to
include those with commutation graph containing no induced $C_4$ or $P_3$. 
In the process, some new embeddings of surface groups into partially 
commutative groups emerge. 
\end{abstract}

\end{frontmatter}

\section{Introduction}

Partially commutative groups are a class of groups widely studied on account of their simple definition, their intrinsically rich structure and their natural appearance in several branches of computer science and mathematics. Crucial examples, which shape the theory of presentations of
groups, arise from study of their subgroups: notably Bestvina and Brady's example of a group which is homologically finite (of type $FP$) but
not geometrically finite (in fact not of type $F_2$); and Mihailova's example of a group with unsolvable subgroup membership problem. More recently results of Wise and others have lead to Agol's proof of the virtual Haken conjecture: that every hyperbolic Haken $3$-manifold
is virtually fibred.  An essential  step in the argument uses the result that the fundamental groups of  so-called ``special'' cube complexes 
embed into partially commutative groups. (It turns out that the manifolds in question have finite covers which are special.)  In the light of such results it is natural to ask which groups arise as subgroups of partially commutative groups, and, in particular, when one partially commutative group embeds in another. 

Let $\Gamma$ be a finite (undirected) simple graph, with vertex set $X$ and edge set $E$,  and let $F(X)$ be the free group on $X$. For elements $g,h$ of a group denote the commutator $g^{-1}h^{-1}gh$ of $g$ and $h$ by $[g,h]$. 
Let
\[
R=\{[x,y]\in F(X) \mid  x,y\in X \textrm{ and }\{x,y\}\in E\}.
\]
Then the {\em partially commutative (pc) group with  commutation graph } $\Gamma$ is the group $\GG(\Gamma)$ with presentation $\left< X\mid R\right>$. 
The complement of a  graph $\Gamma$ is the graph $\overline{\Gamma}$ with the same vertex
set as $\Gamma$ and an edge joining vertices $u$ and $v$ if and only 
if there is no such edge in $\Gamma$. The partially commutative 
group $\GG(\Gamma)$ is said to have \emph{non-commutation  graph} $\overline{\Gamma}$.

Note that the graph $\Gamma$ uniquely determines the partially commutative group up to isomorphism; that is two pc groups $\GG(\Gamma)$ and $\GG(\Lambda)$ are isomorphic if and only if 
$\Gamma$ is isomorphic to $\Lambda$, see \cite{D87.3}. Since isomorphism of pc groups can be characterised in terms of their defining graphs is natural to ask for an analogous characterisation of embedding between pc groups.

\begin{question}[Problem 1.4 in \cite{CSS}]\label{quest1}
Is  there a natural graph theoretic condition which determines
 when  one  partially commutative group embeds in another?
\end{question}

If $\L$ and $\G$ are graphs
such that  $\L$ is isomorphic to a full subgraph of $\G$,
then $\L$ is said to be an \emph{induced subgraph} of 
$\G$ and we write $\L\le \G$.  If $\L$ is not an induced subgraph of $\G$ then $\G$ is said to be $\L$\emph{-free}.
It is easy to see that if $\Lambda$ is an induced subgraph of $\G$ then $\GG(\Lambda)\le \GG(\Gamma)$. 
However, unless $\G$ is a complete graph, $\GG(\G)$ will contain subgroups which
do not correspond to induced subgraphs in  this way.  Moreover, in general pc
groups contain subgroups which are not pc groups. In fact 
Droms proved \cite{D} that  every subgroup of a partially commutative group $\GG(\Gamma)$ is partially commutative if and only if $\Gamma$ is both $C_4$ 
and  $P_3$ free, where $C_n$ is the cycle graph on $n$ vertices 
and $P_n$ is the path graph with $n$ edges. 

Significant progress towards answering  Question \ref{quest1} has been made by Kim and Koberda \cite{KK}, exploiting  
the notion of the  extension of a graph. 
The \emph{extension graph} $\Gamma^e$ of a graph $\Gamma$ is the
graph with vertex set $V^e=\{g^{-1}xg\in \GG(\Gamma)\,|\, x\in V(\Gamma), g\in \GG(\Gamma)\}$ and an edge joining $u$ to $v$ if and only if $[u,v]=1$ in $\GG(\Gamma)$. 
In \cite{KK} it is shown that if $\L$ is a subgraph of $\G^e$ 
then $\GG(\L)$ embeds in $\GG(\G)$. However, the converse is shown to hold only 
 if $\G$ is triangle-free ($C_3$ free): in which case  if   
$\GG(\L)$ embeds in $\GG(\G)$, then $\L$ 
is a subgraph of $\G^e$ (\cite[Theorem 10]{KK}).  This suggests the
following conjecture. 

\begin{conj}
{The Extension Graph Conjecture}, %
{\cite[Conjecture 4]{KK}}%
]
\label{conj:egc}
Let $\Gamma_1$ and $\Gamma_2$ be finite graphs. Then  $\GG(\Gamma_1)< \GG(\Gamma_2)$ if and only if $\Gamma_1< \Gamma_2^e$.
\end{conj}

In Section \ref{sec:Degc} we extend the class of graphs for which
 Conjecture 1 is known to  hold, to include in addition to the triangle-free 
graphs, a family  lying at the opposite end of the spectrum;  
being ``built out of triangles''. 
 More precisely, we generalise one direction of Droms' theorem to prove that if $\G$ is both $C_4$ and $P_3$ free, then the Extension Graph Conjecture holds for $\GG(\G)$. However, in Section \ref{sec:egc} below we give a counterexample to the Extension Graph Conjecture. We note that Droms also showed that a pc group is coherent if and only if it is chordal \cite{D2}, i.e. $C_n$-free, for all $n \ge 4$. Our 
example in Section \ref{sec:egc}, shows that there are chordal graphs, and so coherent pc groups, for which the Extension Graph Conjecture fails.

As mentioned above, one reason that so much attention has been given to pc groups is that several classes of important groups embed into them. 
For example, motivated by 3-manifold theory, there has been a body of research investigating which partially commutative groups contain  closed hyperbolic surfaces as subgroups (that is, fundamental groups of closed surfaces of Euler characteristic less than $-1$; see \cite{CSS} and the references there).
In  \cite{DSS} it was shown that many orientable surface groups embed into pc groups and  subsequently Crisp and Wiest \cite{CW} showed 
that the fundamental group of every compact surface embeds into some pc group, except in  the cases of the compact, closed, non-orientable, surfaces of Euler characteristic $-1, 0$ and $1$, which do not embed in any pc group. As observed in \cite{R}, every orientable surface group in fact embeds in $\GG(C_5)$.  In Section \ref{sec:wcc} we establish similar results for the complement $\overline P_7$ of $P_7$, in place of $C_5$. 

Although  being an induced subgraph is not a necessary condition for embeddability,  for some particular induced subgraphs it is a useful criterion. For instance,  a pc group $\GG(\Gamma)$ contains $\GG(P_3)$ (resp. $C_4$), if and only if $P_3$ (resp. $C_4$) is an induced subgraph of $\Gamma$, see \cite{Kambites,KK}. In \cite{GLR}, the authors ask if the presence of an induced $n$-cycle, $n\ge 5$, in the defining graph $\Gamma$ is a necessary condition for the group $\GG(\Gamma)$ to contain a hyperbolic surface group. However, counterexamples were given in \cite{K08} and 
\cite{CSS}. For example, the fundamental group of the closed, compact,
orientable surface of genus $2$ embeds in the pc group $\GG(P_2(6))$, 
 although $P_2(6)$ has no induced $n$-cycle, for $n\ge 5$. 
Note however that in the case when the graph $\Gamma$ is triangle-free, 
$\GG(C_n) \leq \GG(\G)$ for some $n \geq 5$ if and only if $C_m \leq \G$ 
for some $5\leq m \leq n$, see \cite[Corollary 48]{KK}.

In spite of these examples,  the question raised by Gordon, Long and Reid survives in the following form.
\begin{question}\label{quest:surface}
Is there a partially commutative group $\GG(\Gamma)$ that contains 
the fundamental group of a closed hyperbolic surface but contains 
no  subgroup isomorphic to $\GG(C_n)$, for  $n\ge 5$?
\end{question}
In this weaker form, the question is still open and emphasises the importance of resolving Question \ref{quest1} in the particular case when $\Lambda$ is a $n$-cycle, $n\ge 5$. 

One approach to Question 2 is to prove that certain classes of graphs give rise
to pc groups containing no subgroups $\GG(C_n)$, $n\ge 5$. 
 A possible candidate is the class of weakly chordal graphs:
 a graph $\Gamma$ is 
called \emph{weakly chordal} if $\Gamma$ does not contain an induced $C_n$ or $\overline C_n$ for $n \ge 5$, where $\overline C_n$ denotes the complement of $C_n$.
Since the extension graph of a weakly chordal graph is weakly chordal, see \cite[Lemma 30]{KK}, one can weaken Conjecture 1 as follows. 

\begin{conj}[{The Weakly Chordal Conjecture}, {\cite[Conjecture 13]{KK}}]
\label{conj:wc}
If $\Gamma$ is a weakly chordal graph, then $\GG(\Gamma)$ does not contain $\GG(C_n)$ for any $n\ge 5$.
\end{conj}
If this were true then, as $P_2(6)$ is weakly chordal, and contains hyperbolic surface groups, we would have a positive answer to Question 2.
However, our second example, in Section \ref{sec:wcc}, gives a  counterexample to the Weakly Chordal Conjecture.

\section{Preliminaries}
\subsection{Partially commutative groups}

By \emph{graph} we mean  simple, undirected graph.  
Let $\Gamma$ be a finite graph with vertices $X$ and let $\GG(\Gamma)$  
be the partially commutative group with commutation graph $\Gamma$. 

If $w$ is a word in the free group on $X$ then we say that 
$w$ is \emph{reduced in} $\GG(\Gamma)$ if $w$ has minimal 
length amongst all words representing the same element of $\GG(\Gamma)$
as $w$. 
If $w$ is reduced in $\GG(\Gamma)$ then 
we define $\az_\G(w)$ to be the set of elements of $X$ such that $x$ or 
$x^{-1}$ occurs in $w$. 
It is well-known that all words which are reduced in $\GG(\G)$ and  
represent a particular
element $g$ of $\GG(\G)$ have the same length, and  that if $w=w'$ in $\GG(\Gamma)$ 
then $\az_\G(w)=\az_\G(w')$. 
  
For a word $w$ in the free monoid on $X\cup X^{-1}$  we 
write $\az(w)$ for the set of elements of $X$ such that $x$ or $x^{-1}$
occurs in $w$. We 
write $u\doteq v$ to indicate that $u$ and $v$ are equal as words
 in the free monoid on $X\cup X^{-1}$.    

\subsection{Van Kampen diagrams}

By van Kampen's Lemma (see \cite{BSR}) 
the word $w$ represents the trivial element in a fixed group $G$ given by the presentation $\langle A \mid R \rangle$ if and only if there exists a finite connected, oriented, based, labelled, planar graph $\mathcal{D}$, where each oriented edge is labelled by a letter in $A^{\pm 1}$, each bounded region 
($2$-\emph{cell}) of $\mathbb{R}^2 \setminus \mathcal{D}$ is labelled by a word in $R$ (up to shifting cyclically or taking inverses) and $w$ can be read on the boundary of the unbounded region of $\mathbb{R}^2 \setminus \mathcal{D}$ from the base vertex. Then we say that $\mathcal{D}$ is a \emph{van Kampen diagram} for the boundary word $w$ over the presentation $\langle A \mid R \rangle$. Note that any van Kampen diagram can also be viewed as a 2-complex, with a $2$-cell attached for each bounded region. A van Kampen diagram for 
the word $w$ is \emph{minimal} if no other van Kampen diagram 
for $w$ has fewer $2$-cells.

Following monograph \cite{O}, if we complete the set of defining relations adding the trivial relations $1\cdot a= a\cdot 1$ for all $a \in A$, then every van Kampen diagram can be transformed so that its boundary is a simple curve. In other words, as a 2-complex the van Kampen diagram is homeomorphic to a disc tiled by $2$-cells which are also homeomorphic to a disc. We shall  assume that all van Kampen diagrams are of this form.

From now on we  shall restrict our considerations to the case when $G$ is a 
partially commutative group. 
Let $\mathcal{D}$ be a minimal van Kampen diagram for the boundary word $w$. 
Given an occurrence of a letter $a\in A^{\pm 1}$ in $w$, there is a $2$-cell 
$C$, in the 2-complex $\mathcal{D}$, attached to $a$. Since every 
$2$-cell in a van Kampen diagram is either labelled by a relation of the 
form $a^{-1}b^{-1}ab$ or is a \emph{padding} $2$-cell, i.e. a $2$-cell 
labelled by $1\cdot a\cdot 1\cdot a^{-1}$, there is just one 
occurrence of $a$ and one occurrence of $a^{-1}$ on the boundary of $C$.

Since $\mathcal{D}$ is homeomorphic to a disc, if the occurrence of $a^{-1}$ on the boundary of $C$ is not on the boundary of $\mathcal{D}$, there exists a unique $2$-cell $C' \ne C$ attached to this occurrence of $a^{-1}$ in $\mathcal{D}$. Repeating this process, we obtain a unique \emph{band} in $\mathcal{D}$.

Because of the structure of the $2$-cells and the fact that $\mathcal{D}$ is homeomorphic to a disc, a band  never self-intersects in a $2$-cell; indeed, since $\mathcal{D}$ is homeomorphic to a disc, a  $2$-cell corresponding to a self-intersection of a band would be labelled by the word $aaa^{-1}a^{-1}$, a contradiction.

Thus, since the number of $2$-cells in $\mathcal{D}$ is finite, in a finite number of steps the band will again meet the boundary in an occurrence of $a^{-1}$ in $w$.

We will use the notation $L_{a}$ to indicate that a band begins (and thus ends) in an occurrence of a letter $a\in A^{\pm 1}$. Fix an orientation of 
the boundary of $\mathcal{D}$ so that the label of the boundary, read
from the base vertex, with this orientation, is the word $w$.  
The band $L_a$ determines a \emph{boundary subword} $w(L_a)$ of $w$: namely the subword of $w$  that is read on the boundary between the ends of $L_a$,
in the direction of fixed orientation. 
In other words, every band $L_a$ induces the following decomposition 
of the boundary word $w$: 
$$
w= w_1 a^\epsilon w(L_a) a^{- \epsilon} w_2,
$$
where $\epsilon \in \{ \pm 1\}$.

Furthermore the band $L_a$ determines a \emph{band word} $z(L_a)$: namely the word obtained by reading along the boundary of the band from the terminal end of $a^\epsilon$ to the initial end of $a^{-\epsilon}$. 

Note that if two bands $L_a$ and $L_b$ cross then the intersection $2$-cell 
realises the equality $a^{-1}b^{-1}ab=1$ and so $a\ne b$ and $[a,b]=1$. 
It follows that if a letter $c$  occurs  in $w(L_a)$ and $c=a^{\pm 1}$ or $[c,a] \ne 1$, then the ends of the band $L_c$ both occur in $w(L_a)$, i.e. the band $L_c$ is contained in the region bounded by $z(L_a)$, and $w(L_a)$. Notice that since $z(L_a)$ labels the boundary of the band $L_a$, it follows that every letter $b \in z(L_a)$ we have that $b\ne a^{\pm 1}$ and $[b,a]=1$.

The following will be useful in the next section. 
If $L_u$ is a band, where $u=a^{\pm 1}$, for some $a\in A$
such that 
$a\notin\az(w(L_u))$, then we say that $L_{u}$ is an 
\emph{outside band}.   
\begin{lem}\label{lem:bands}
Let $\mathcal{D}$ be a 
van Kampen diagram for the boundary word $w$.
\begin{enumerate}
\item\label{it:bands1} If $x\in\az(w)$ then there exists an outside band $L_u$, $u=x^{\pm 1}$. 
\item\label{it:bands2} Let $L_v$ be a band, where $v\in A^{\pm 1}$. If   $y\in \az(w(L_v))$, such that
$[y,v]\neq 1$ or $y=v^{\pm 1}$, then there exists an outside band $L_z$, 
with $z= y^{\pm 1}$, which has  both ends in $w(L_v)$.  
\end{enumerate}
\end{lem}
\begin{proof}
\begin{enumerate}
\item Amongst all bands beginning and ending in $x^{\pm 1}$, choose a band $L_u$, $u=x^{\pm 1}$, such that the length of $w(L_u)$ is minimal. As no two
bands with ends   $x^{\pm 1}$ can cross it follows that $x\notin \az(w(L_u))$. 
Hence $L_u$ is an outside band.
\item As $y\in \az(w(L_v))$, there exists a band $L$ with one end at an
occurrence of $y^{\pm 1}$ in $w(L_v)$. As $[y,v]\neq 1$ or $y=v^{\pm 1}$,  
bands  $L_v$ and $L$ cannot cross, so $L$ has both ends in $w(L_v)$. That is,
every band with one end at   an
occurrence of $y^{\pm 1}$ in $w(L_v)$ has both ends in $w(L_v)$. Now
choose such a band $L'$ with $|w(L')|$ minimal over lengths of boundary
words of all
bands ending at occurrences of $y^{\pm 1}$ in $w(L_v)$. As in the proof of the
 first 
part of the Lemma, $L'$ is an outside band.
\end{enumerate}
\end{proof}

\section{A counterexample to the Extension Graph Conjecture}\label{sec:egc}

In this Section we give a counterexample to Conjecture \ref{conj:egc}.
Let $\Gamma_1$ and $\Gamma_2$ be the graphs of Figure \ref{fig:1a} and 
Figure \ref{fig:1b}, respectively. 
\begin{figure}
\psfrag{a}{$a$}
\psfrag{a1}{$a_1$}
\psfrag{a2}{$a_2$}
\psfrag{b}{$b$}
\psfrag{c}{$c$}
\psfrag{d}{$d$}
\psfrag{e}{$e$}
\begin{center}
\begin{subfigure}[b]{.45\columnwidth}
\begin{center}
\includegraphics[scale=0.3]{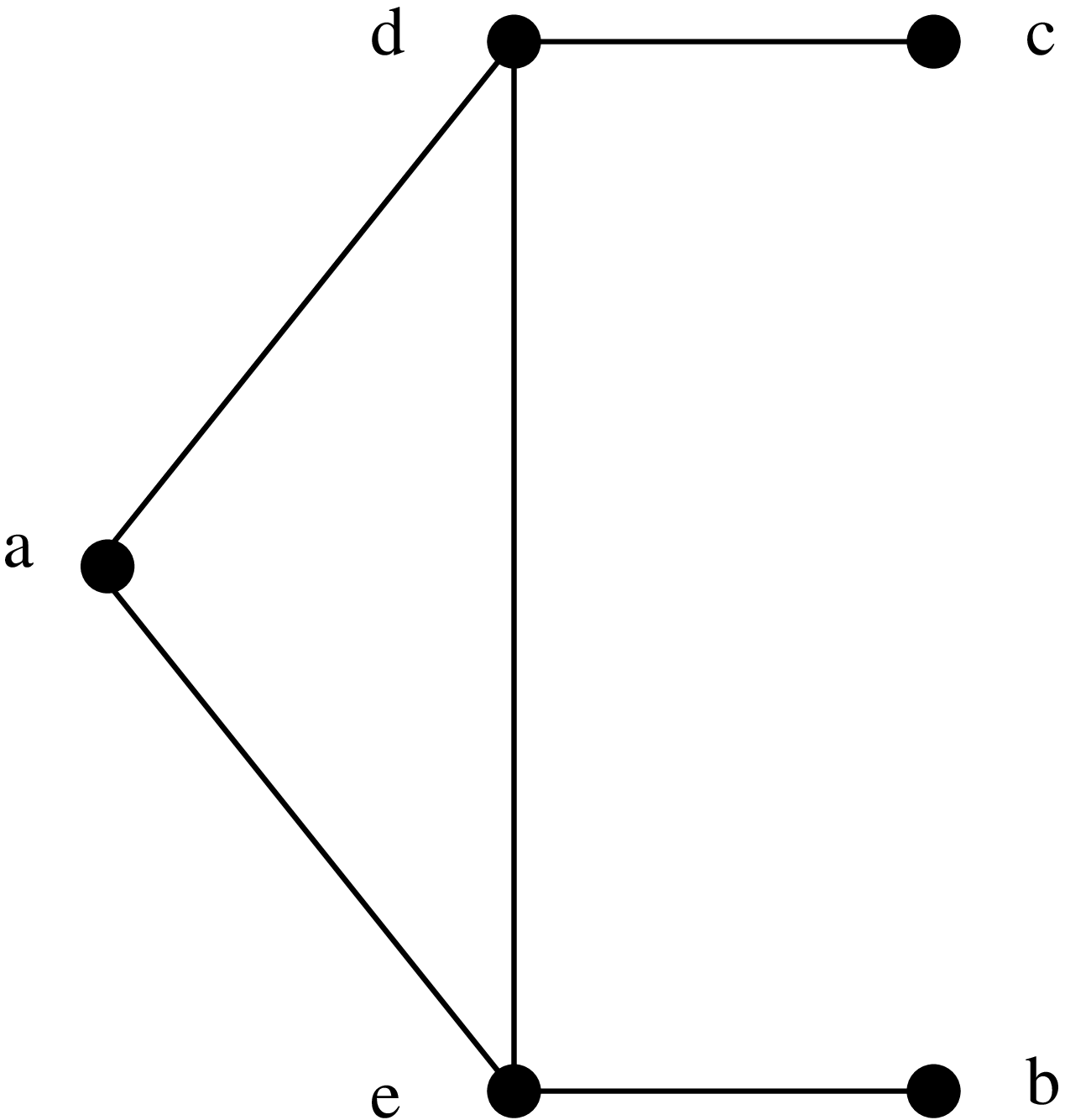}
\caption{$\Gamma_1$}
\label{fig:1a}
\end{center}
\end{subfigure}
\begin{subfigure}[b]{.45\columnwidth}
\begin{center}
\includegraphics[scale=0.3]{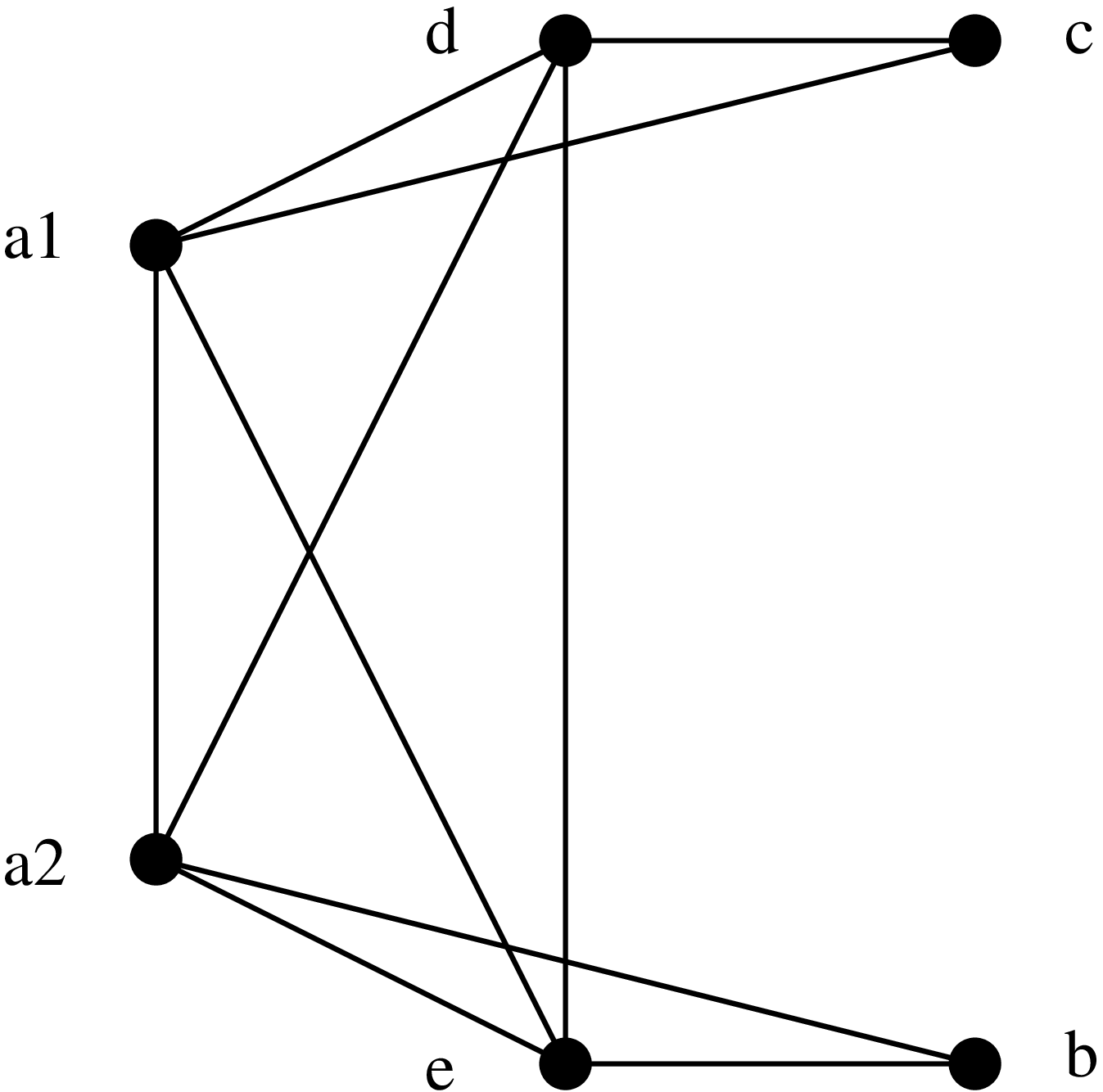}
\end{center}
\caption{$\Gamma_2$}
\label{fig:1b}
\end{subfigure}
\end{center}
\caption{}\label{fig:1}
\end{figure}
Let $\GG(\Gamma_1)$ and $\GG(\Gamma_2)$ be the corresponding partially commutative groups: 
$$
\begin{array}{ll}
\GG(\Gamma_1) =& \langle a, b,c,d,e\mid [a,d], [a,e], [b,e], [c,d], [d,e]\rangle, \\
\GG(\Gamma_2) =& \langle a_1, a_2, b, c, d,e \mid [a_1,a_2], [a_1, c], [a_1,d],[a_1,e], [a_2,b],[a_2,d],[a_2,e], [b,e],[c,d],[d,e]\rangle.
\end{array}
$$

Define a map
$$
\varphi =\left\{ 
\begin{array}{l}
a\mapsto a_1a_2;\\
b\mapsto b;\\
c\mapsto c;\\
d\mapsto d;\\
e\mapsto e.
\end{array}
\right.
$$

\begin{lem} \label{lem:1}
The map $\varphi$ defined above extends to a monomorphism $\varphi: \GG(\Gamma_1)\to \GG(\Gamma_2)$.
\end{lem}
\begin{proof}
It is immediate that 
$$
\varphi([a,d])=\varphi([a,e])=\varphi([b,e])=\varphi([c,d])=\varphi([d,e])=1
$$
in $\GG(\Gamma_2)$, hence $\varphi$ extends to a homomorphism. 

We show that $\varphi$ is injective. 
For a contradiction assume that $\ker(\varphi)$ contains a non-trivial element $g$ and let $w=w(a,b,c,d,e)\in \GG(\Gamma_1)$ be 
a minimal length word representing $g$. 
Consider  a minimal van Kampen diagram $\mathcal D$ for $\varphi(w)=w(a_1a_2,b, c,d,e)=1$.

Since the canonical parabolic subgroup generated by $\{b,c,d,e\}$ is isomorphic to its image by $\varphi$, it follows that the word $w$ must contain at least one occurrence of the letter $a$ (and its inverse $a^{-1}$). 
Since $\varphi(a)=a_1a_2$, $\varphi(w)$ contains at least one 
occurrence of $a_1$ and of $a_2$ and every occurrence of $a_1^{\pm 1}$ 
(or $a_2^{\pm 1}$) in $\varphi(w)$ occurs in a subword $(a_1a_2)^{\pm 1}$.
Moreover, if $w$ has  subword of the form 
$(a_1a_2)^\epsilon u(a_1a_2)^{-\epsilon}$, with $\epsilon \in\{\pm 1\}$ and
$a_1\notin \az(u)$, then it follows that $a_2\notin \az(u)$ (and if 
$a_2\notin \az(u)$ then $a_1\notin \az(u)$). 

We claim that if $L_x$ is an outside band, with $x=a_1^{\pm 1}$, then
 there exists a band $L$ with one end at an occurrence of $c^{\pm 1}$
in $w(L_x)$ and its other end not in  $w(L_x)$. Similarly, if 
$L_y$ is an outside band, with $y=a_2^{\pm 1}$, then
 there exists a band $L$ with one end at an occurrence of $b^{\pm 1}$
in $w(L_y)$ and its other end not in  $w(L_y)$. As there is an isomorphism
of $\G_2$ interchanging $a_1$ and $a_2$ and $b$ and $c$, it suffices
to prove the first of these claims. 

To see that the first claim holds, suppose  that 
$L_x$ is an outside band, with $x=a_1^{\pm 1}$. Without loss of generality
(replacing $w$ with $w^{-1}$ if necessary), we 
may assume that $x=a_1$. Then the band $L_x$ induces the 
 decomposition 
$\varphi(w)\doteq w_1 a_1 a_2 w_2 a_2^{-1} a_1^{-1} w_3$ 
and,
 since $a_1\notin \az(w_2)$, it follows that   
 $a_1,a_2\notin \az(w_2)$.  Then $\az(w_2)$ must contain either $b$ or $c$.  
Indeed, otherwise, since $w_2$ belongs to the subgroup generated by 
$\{b,c,d,e\}$ and so $w_2=\varphi(w_2)$, the subword $a w_2 a^{-1}$ 
of $w$ would not be reduced in $\GG(\Gamma_1)$, a contradiction. 
Also, as $w_2=\varphi(w_2)$, the word $w_2$ is reduced in $\GG(\G_2)$. 

Assume  to begin with that $\az(w_2)$ does not contain $c$; so does contain 
 $b$. In this case, since $[a_1,b]\neq 1$, 
Lemma \ref{lem:bands}.\ref{it:bands2} implies the existence of an outside band 
$L_z$, where $z=b^{\pm 1}$, with both ends in $w_2$. As $L_z$ is an outside band
$b\notin \az(w(L_z))$ and hence, since $w(L_z)$ is a subword of $w_2$,
we have $\az(w(L_z))\subseteq \{d,e\}$. Moreover, 
$w_2$ is reduced in $\GG(\G_2)$   
so $w(L_z)$ must contain an occurrence $d^{\pm 1}$. 
From Lemma \ref{lem:bands}.\ref{it:bands2} again, there is an outside band
$L_{v}$, where $v= d^{\pm 1}$, with both ends in $w(L_z)$. However, 
this means that $\az(w(L_{v}))\subseteq
\{e\}$, and so $w_2$ is not reduced  in $\GG(\G_2)$, a contradiction.

Therefore we conclude that  $c\in \az(w_2)$
and there is a band $L_{u}$, where $u=c^{\pm 1}$, with one
end in $w_2$. If both ends of $L_u$ are in $w_2$ then
there exists an outside band $L_z$, $z=c^{\pm 1}$, with both ends in
$w_2$ (Lemma \ref{lem:bands}.\ref{it:bands2}). As $w_2$ is reduced 
in $\GG(\G_2)$,
$\az(w(L_z))$ must contain either $b$ or $e$. If $b\in \az(w(L_z))$ then,
since $c\notin \az(w(L_z))$, we obtain a contradiction, using the 
argument of the previous paragraph. Therefore $b\notin \az(w(L_z))$
and $e\in \az(w(L_z))$. This implies the existence of an outside
band $L_v$, $v=e^{\pm 1}$, with both ends in $w(L_z)$. However, 
in this case $\az(w(L_v))\subseteq \{d\}$, and again $w_2$ is
not reduced in $\GG(\G_2)$, a contradiction. 
Thus one end of $L_{u}$ lies in $w_1$
or $w_3$. This proves the claim, for $L_x$, with $L=L_u$.  

 Returning to the proof of the lemma, let $L_{x_1}$ be an 
outside band, with $x_1=a_1^{\pm 1}$, and let 
$w=w_1(a_1a_2)^\epsilon w_2(a_1a_2)^{-\epsilon}w_3$, be
the corresponding decomposition of $w$, where 
$w(L_{x_1})=w_2$, if $\epsilon =1$ and $w(L_{x_1})=a_2^{-1}w_2a_2$, otherwise.
From the claim above, there exists a band $L_{c_1}$, where $c_1=c^{\pm 1}$, 
 with one end in $w_2$
and the other end in $w_1$ or $w_3$.  Taking a cyclic permutation of $w$ (and $\varphi(w)$), if 
necessary, we may assume that $L_{c_1}$ has its other end in $w_1$.

Let $L'$ be the band which has one end at the occurrence
of $a_2^{\pm 1}$ in $(a_1a_2)^{\epsilon}w_2$. 
Since $[a_2,c]\neq 1$ the bands $L_{c_1}$ and $L'$ cannot cross and, 
as $a_2\notin \az(w_2)$, the other end of $L'$ must lie in $w_1$ 
(see Figure \ref{fig:2}). Now we find an outside band $L_{y_1}$, 
where $y_1=a_2^{\pm 1}$, with both ends in $w(L')$. Hence $w_1$ decomposes as 
\[
w_1\doteq w_{11}c^{-\delta}w_{12}a_2^{-\epsilon}w_{13}a_2^{\gamma} w_{14}a_2^{-\gamma}w_{15},
\]
for some $\gamma,\delta\in\{\pm 1\}$, as shown in 
  Figure \ref{fig:2}. 
\begin{figure}
\psfrag{w11}{$w_{11}$}
\psfrag{w12}{$w_{12}$}
\psfrag{w13}{$w_{13}$}
\psfrag{w14}{$w_{14}$}
\psfrag{w15}{$w_{15}$}
\psfrag{Lx}{$L_{x_1}$}
\psfrag{Ly}{$L_{y_1}$}
\psfrag{Lc}{$L_{c_1}$}
\psfrag{L'}{$L'$}
\begin{center}
\includegraphics[scale=0.55]{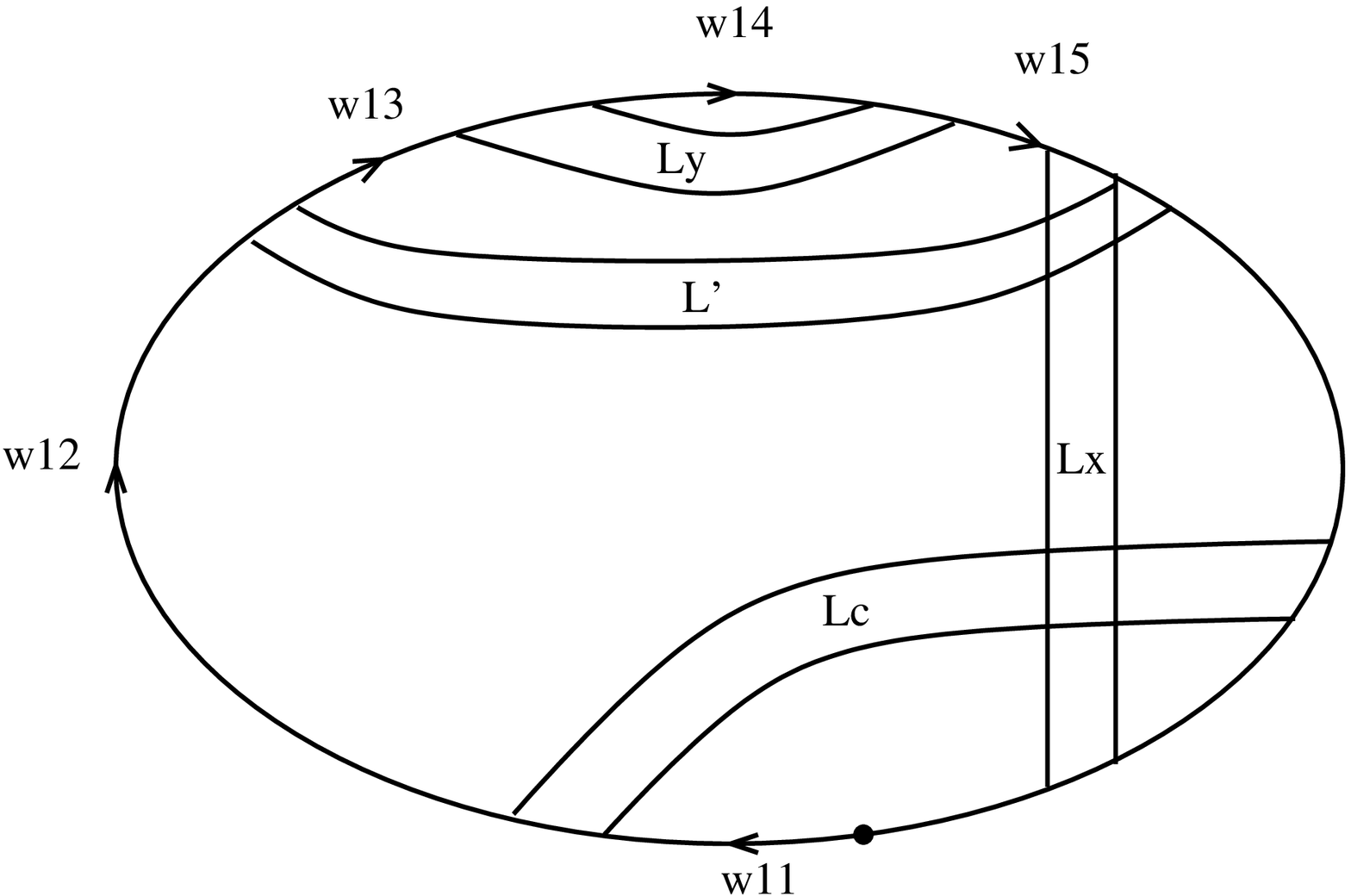}
\caption{}\label{fig:2}
\end{center}
\end{figure}
Thus, the claim above implies  
 there exists a band 
$L_{b_1}$, where $b_1=b^{\pm 1}$, with exactly one end in $w(L_y)=w_{14}$. 
As $[b,c]\neq 1$ and $[b,a_1]\neq 1$ the band $L_{b_1}$ cannot 
cross  $L_{c_1}$ or $L_x$,  so the 
other end of $L_{b_1}$ lies in  
in $w_{12}$, $w_{13}$ or 
$w_{15}$. In particular $w(L_{b_1})$ is a subword of $w(L_{c_1})$.
Moreover, since $w(L_{b_1})$ contains an occurrence of $a_2^{\pm 1}$ we 
have $a_1,a_2\in \az(w(L_{b_1}))$. 

Now suppose that we have occurrences of letters 
$c_1,\ldots ,c_n$, $b_1,\ldots , b_n$, with $c_i=c^{\pm 1}$ and 
$b_i=b^{\pm 1}$, and corresponding bands $L_{c_i}$, $L_{b_i}$, such
that $w(L_{b_i})$ is a subword of $w(L_{c_i})$, for $i=1,\ldots, n$,  
$w(L_{c_{i+1}})$ is a subword of $w(L_{b_i})$, for $i=1,\ldots, n-1$,  
and $a_1,a_2\in w(L_{b_n})$. 

 As $[a_1,b]\neq 1$ there exists an outside band $L_{x_{n+1}}$, where 
$x_{n+1}=a_1^{\pm 1}$, with both ends in $w(L_{b_n})$. Thus 
$w(L_{b_n})$ decomposes as 
\[w(L_{b_n})\doteq u_1(a_1a_2)^{\zeta}u_2(a_1a_2)^{-\zeta}u_3,
\]
where $L_{x_{n+1}}$ has both ends at occurrences of $a_1^{\pm 1}$ in 
$(a_1a_2)^{\zeta}u_2(a_1a_2)^{-\zeta}$, 
for some $u_2\in \langle b,c,d,e\rangle$ and $\zeta \in\{\pm 1\}$. From the claim above, there is a band 
$L_{c_{n+1}}$, where $c_{n+1}=c^{\pm 1}$, with one end in $u_2$ and 
the other end not in $u_2$. As $[b,c]\neq 1$, the other end of  
 $L_{c_{n+1}}$ lies in $u_1$ or $u_3$. In particular, $w(L_{c_{n+1}})$
 is a subword of $w(L_{b_n})$. 

If  $L_{c_{n+1}}$ has one end in $u_1$ then let $L''$ be the
band which has one end at the occurrence of $a_2^{\pm 1}$ in 
$(a_1a_2)^{\zeta}u_2$, as shown in 
  Figure \ref{fig:vk1}. 
\begin{figure}
\psfrag{Lb}{$L_{b_n}$}
\psfrag{Lbn}{$L_{b_{n+1}}$}
\psfrag{Lcn}{$L_{c_{n+1}}$}
\psfrag{Lx}{$L_{x_{n+1}}$}
\psfrag{Ly}{$L_{y_{n+1}}$}
\psfrag{Lc}{$L_{c_1}$}
\psfrag{L'}{$L''$}
\psfrag{...}{$\cdots$}
\begin{center}
\includegraphics[scale=0.55]{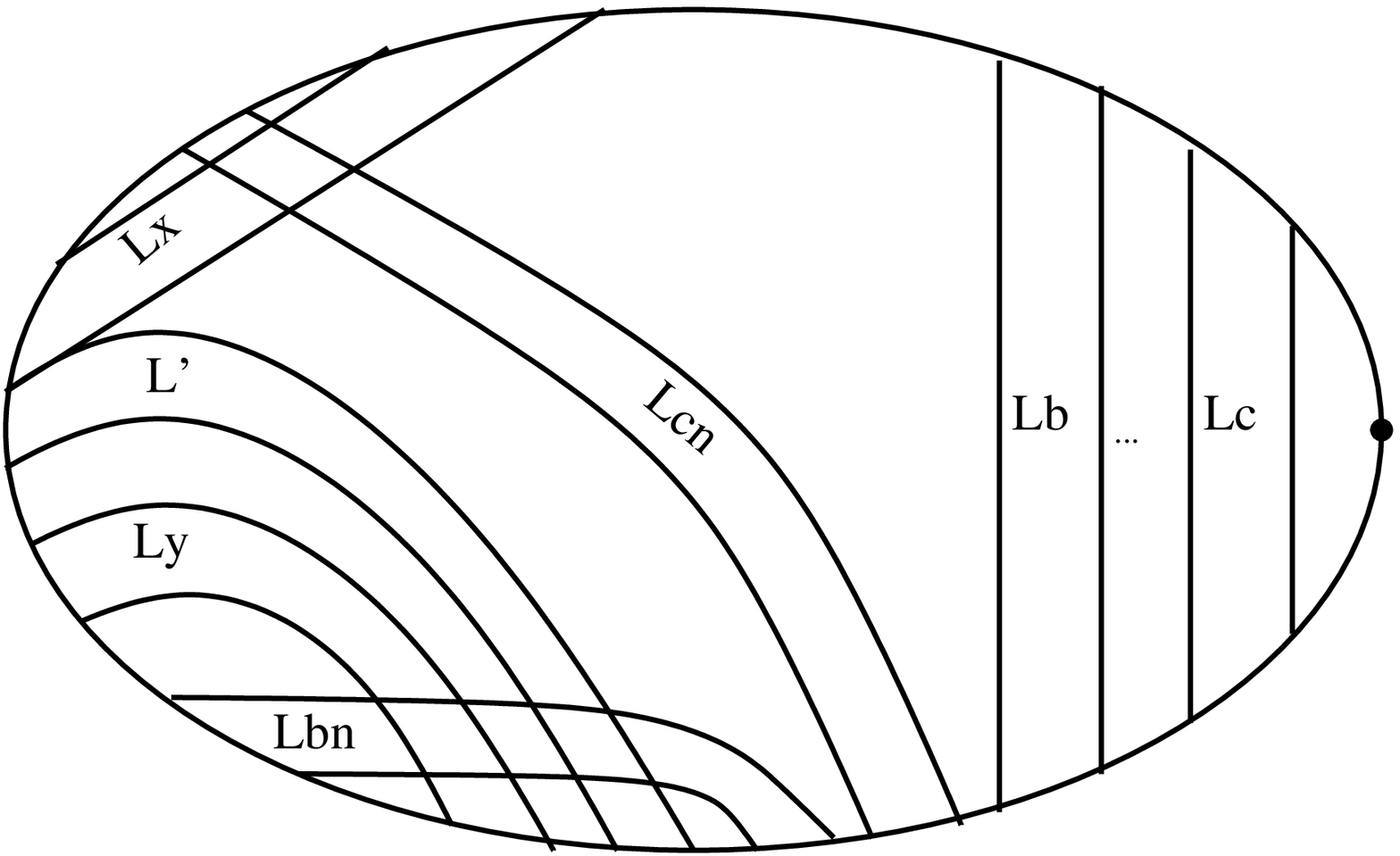}
\caption{}\label{fig:vk1}
\end{center}
\end{figure}
Then $w(L'')$ is a subword of $u_1(a_1a_2)^{\zeta}$ and there exists
an outside band $L_{y_{n+1}}$, where $y_{n+1}=a_2^{\pm 1}$, with both
ends in $w(L'')$. 

From the claim again, we then have a band 
 $L_{b_{n+1}}$, where $b_{n+1}=b^{\pm 1}$, with one end in 
$w(L_{y_{n+1}})$ and the other end outside $w(L_{y_{n+1}})$. As $[b,c]\neq 1$, it 
follows that the other end of $L_{b_{n+1}}$ lies in $w(L_{c_{n+1}})$, 
so $w(L_{b_{n+1}})$ is a subword of $w(L_{c_{n+1}})$. 

If  $L_{c_{n+1}}$ has one end in $u_3$ then let $L''$ be the
band which has one end at the occurrence of $a_2^{\pm 1}$ in 
$u_2(a_1a_2)^{-\zeta}$;
see Figure \ref{fig:4}. 
\begin{figure}
\psfrag{Lb}{$L_{b_n}$}
\psfrag{Lbn}{$L_{b_{n+1}}$}
\psfrag{Lcn}{$L_{c_{n+1}}$}
\psfrag{Lx}{$L_{x_{n+1}}$}
\psfrag{Ly}{$L_{y_{n+1}}$}
\psfrag{Lc}{$L_{c_1}$}
\psfrag{L'}{$L''$}
\psfrag{...}{$\cdots$}
\begin{center}
\includegraphics[scale=0.55]{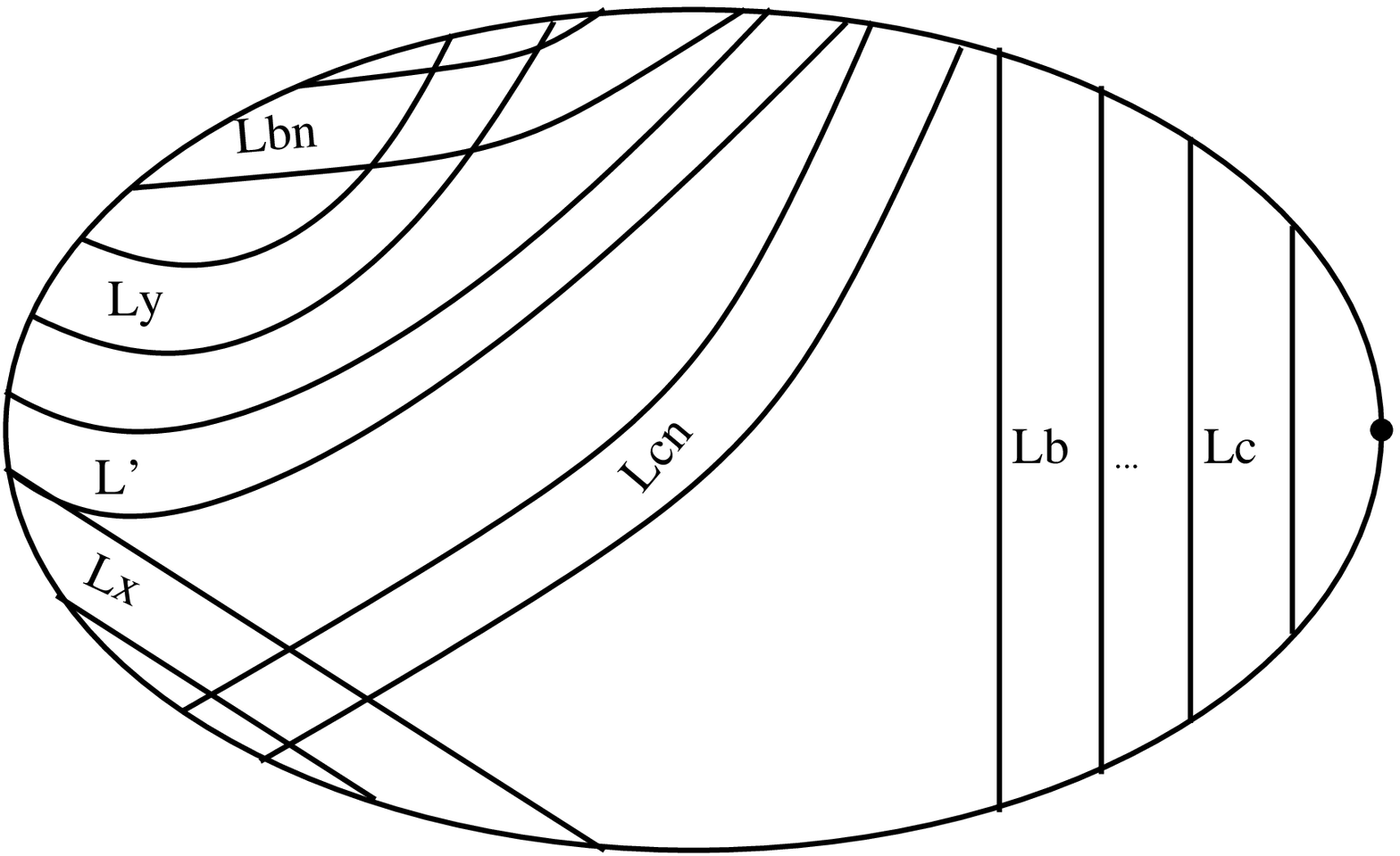}
\caption{}\label{fig:4}
\end{center}
\end{figure}
 A similar argument shows that we can 
find a band $L_{y_{n+1}}$, where $y_{n+1}=a_2^{\pm 1}$, with both
ends in $w(L'')$, and then a band $L_{b_{n+1}}$,
 where $b_{n+1}=b^{\pm 1}$, with exactly one end in 
$w(L_{y_{n+1}})$ and the other end in $w(L_{c_{n+1}})$, so in this case
as well,  $w(L_{b_{n+1}})$ is a subword of $w(L_{c_{n+1}})$. 
As only one
end of $L_{b_{n+1}}$ lies in $w(L_{y_{n+1}})$ it follows, as above, 
that $a_1,a_2\in \az(w(L_{b_{n+1}}))$. Hence we can extend the 
list above, of occurrences $c_i$ and $b_i$, to $c_{n+1}$ and $b_{n+1}$. 

Therefore, if such a diagram exists, there is an infinite sequence of
subwords $w(L_{c_i})$ of $\varphi(w)$, no two of which are equal, such that 
$w(L_{c_{i+1}})$ is a subword of  $w(L_{c_i})$. Since $\varphi(w)$ 
is of finite length, this is a contradiction. 
\end{proof}

\begin{lem} \label{lem:2}
The graph $\Gamma_1$ is not an induced subgraph of the extension 
graph $\Gamma_2^e$.
\end{lem}
\begin{proof}
Suppose, for a contradiction, that $\G_1$ is an induced subgraph of $\G_2^e$. 
Then there are vertices $v_a, v_b, v_c, v_d, v_e, v_f$ of $\G_2^e$, such that the map sending $a$ to $v_a$, $b$ to $v_b$ and so on, induces
an embedding of graphs. Therefore, by definition of $\G_2^e$ there exist canonical generators $x_a,x_b,x_c,x_d, x_e,\in \{a_1,a_2,b,c,d,e\}$ and elements 
$w_a,w_b,w_c,w_d,w_e\in \GG(\Gamma_2)$, such that 
\[
 v_a=x_a^{w_a},v_b=x_b^{w_b}, v_c=x_c^{w_c}, v_d=x_d^{w_d}, v_e=x_e^{w_e},
\]
where  the words $w_y^{-1}x_yw_y$, are reduced in $\GG(\G_2)$, 
for $y\in \{a,b,c,d,e,f\}$ and the graph spanned by the set $\{v_a, v_b, v_c, v_d, v_e\}$ in $\G_2^e$ is isomorphic to the graph $\G_1$, i.e. the following commutation relations, and only these, 
hold between the $v_y$:
\[[v_a,v_d], [v_a,v_e], [v_b,v_e], [v_c,v_d], [v_d,v_e].
\]

We perform a case-by-case analysis. 
We repeatedly use the fact that if $Y$ is a subset of the canonical generating
set of a pc group $\GG(\G)$ and $g\in \langle Y\rangle$ then 
$\az_{\G}(g)\subseteq Y$. That is,
if $w$ is a reduced word in $\GG(\G)$ and represents $g$ 
then every letter occurring 
in $w$ belongs to $Y$. In particular, if $x$ is a generator and $v$ is 
a reduced word in $\GG(\G)$  then $x\in \az_{\G}(x^v)$.

Note that the sets of vertices $\{v_a, \dots, v_e\}$ and $\{v_a^w, \dots, v_e^w\}$ in $\G_2^e$, where $w\in \GG(\G_2)$, span isomorphic graphs, in this case both isomorphic
to  $\G_1$. 
Hence, conjugating the vertices of the set $\{v_a, \dots, v_e\}$ by $w_d^{-1}$ we may assume that $w_d=1$.

\paragraph{Case I: $x_d=c$} In this case,  using the Centraliser Theorem, \cite{Serv, DK}, since $[a,d]=1$, it follows that 
$x_a^{w_a}\in C_{\GG(\G_2)}(x_d)=\langle a_1, c, d\rangle$, a free Abelian group, so 
$w_a=1$ and $x_a\in \{ a_1,d\}$.

Similarly, since $[e,d]=[c,d]=1$, it follows that $x_c,x_e\in \{a_1,d\}$
 and $w_c,w_e=1$. It is clear that $x_a, x_e$ and $x_c$ need to be pairwise distinct,
so this is  a contradiction.
 
\paragraph{Case II: $x_d=b$} This case follows from Case I, using the symmetry of the graph $\Gamma_2$ (interchanging $b$ and $c$).

\paragraph{Case III: $x_d=d$} From the Centraliser Theorem again,
$C_{\GG(\G_2)}(d)=\langle a_1,a_2, c,d,e \rangle$,  which has 
 centre  
$\langle a_1,d\rangle$. 
Since $[d,e]=1$, we have $x_e^{w_e}\in \langle a_1,a_2, c,d,e \rangle$,  
and so, as $w_e^{-1}x_ew_e$ is reduced in $\GG(\G_2)$, the 
letters $a_1$ and $d$ do not belong to $\az(w_e)$ and   
$w_e\in \langle a_2,c,e\rangle$. 
This means that $w_e\in C_{\GG(\G_2)}(x_d)$, so we may conjugate again, 
this time 
by $w_e^{-1}$, and assume that $w_e=1$.
As $x_d$ and $x_e$ must be distinct, we must have   
$x_e\in\{a_1,a_2,c,e\}$.  If $x_e=c$, then 
the statement follows from Case I using the symmetry of $\G_2$ 
(interchanging $d$ and $e$). 

Hence we may assume that  $x_e\in\{a_1,a_2,e\}$.

\paragraph{Case III.1: $x_e=a_1$} In this case $x_e\in Z(C_{\GG(\G_2)}(d))$. 
As above, since $[c,d]=1$, it follows that 
$x_c^{w_c}\in C_{\GG(\Gamma_2)}(d)$ and so 
$[v_e,v_c]=[x_e,x_c^{w_c}]=1$, a contradiction. 

\paragraph{Case III.2: $x_e\in\{a_2,e\}$} 
As above, since $[c,d]=[a,d]=1$, it follows that $x_c,x_a\in \{a_1,a_2,c,e\}$. If $x_c=a_1$, then  $[v_c,v_e]=[x_c^{w_c},x_e^{w_e}]=1$, a contradiction.  
It follows that $x_c\in \{a_2,c,e\}$. Since $[c,a]\ne 1$, it follows that 
$x_a\in \{a_2,c,e\}$.   
On the other hand, since $[a,e]=1$, it follows 
 that 
$[v_a,v_e]=[x_a^{w_a},x_e]=1$, so $x_a^{w_a}\in C_{\GG(\G_2)}(\{a_2,e\})
=C_{\GG(\G_2)}(e)=\langle a_1,a_2,e,d,b\rangle$. Thus 
$x_a\in \{a_2,e\}$  and $w_a=1$. However, this means that
$C_{\GG(\G_2)}(v_a)=C_{\GG(\G_2)}(\{a_2,e\})=C_{\GG(\G_2)}(v_e)$,
 contrary to the hypotheses on the $v_y$'s.
 
\paragraph{Case IV: $x_d=e$} Follows from Case III using symmetry of the graph $\Gamma_2$ (interchanging $e$ and $d$).

\paragraph{Case V: $x_d=a_1$}
This follows from Case III using the symmetry of $\Gamma_2$ (interchanging
$a_1$ and $d$). 

\paragraph{Case VI: $x_d=a_2$} Follows from Case IV using symmetry of the graph $\Gamma_2$.

\end{proof}

\section{A counterexample to the Weakly Chordal Conjecture}
\label{sec:wcc} 

In this section we give a counterexample to Conjecture \ref{conj:wc}. 
\begin{rem}
We note that the graph $\Gamma$ is weakly chordal if and only if  the complement 
graph $\overline{\Gamma}$ is weakly chordal. 
\end{rem}

Let $\Gamma_1$ be the graph $C_5$ (the 5-cycle)and $\Gamma_2$ the graph 
$P_7$ (the path of length 7), as shown in Figure \ref{fig:3}.  
 Let $\GG(\overline C_5)$ and $\GG(\overline P_7)$ be partially commutative groups with the underlying non-commutation graphs $\Gamma_1$ and $\Gamma_2$, respectively. 
That is 
\begin{equation}\label{eq:c5}
\GG(\overline C_5)=\langle a,b,c,d,e\mid [a,c],[a,d],[b,d],[b,e],[c,e]\rangle
\end{equation}
and
\begin{align*}
\GG(\overline P_7)=\langle a,b,c_1,c_2,d_1,d_2, e_1,e_2 
\mid &[a,c_1],[a,c_2],[a,d_1],[a,d_2], [a,e_2],\\
& [b,c_1],[b,d_1],[b,d_2], [b, e_1], [b,e_2],\\
& [c_1,c_2],[c_1,d_2], [c_1,e_1], [c_1,e_2],\\ 
& [c_2,d_1], [c_2,e_1],[c_2,e_2],\\ 
&[d_1,d_2], [d_1,e_2], [d_2,e_1],\\ 
& [e_1,e_2] \rangle.
\end{align*}
\begin{figure}
\psfrag{a}{$a$}
\psfrag{c1}{$c_2$}
\psfrag{c2}{$c_1$}
\psfrag{b}{$b$}
\psfrag{c}{$c$}
\psfrag{d}{$d$}
\psfrag{e}{$e$}
\psfrag{d1}{$d_1$}
\psfrag{d2}{$d_2$}
\psfrag{e1}{$e_1$}
\psfrag{e2}{$e_2$}
\begin{center}
\begin{subfigure}[b]{.45\columnwidth}
\begin{center}
\includegraphics[scale=0.25]{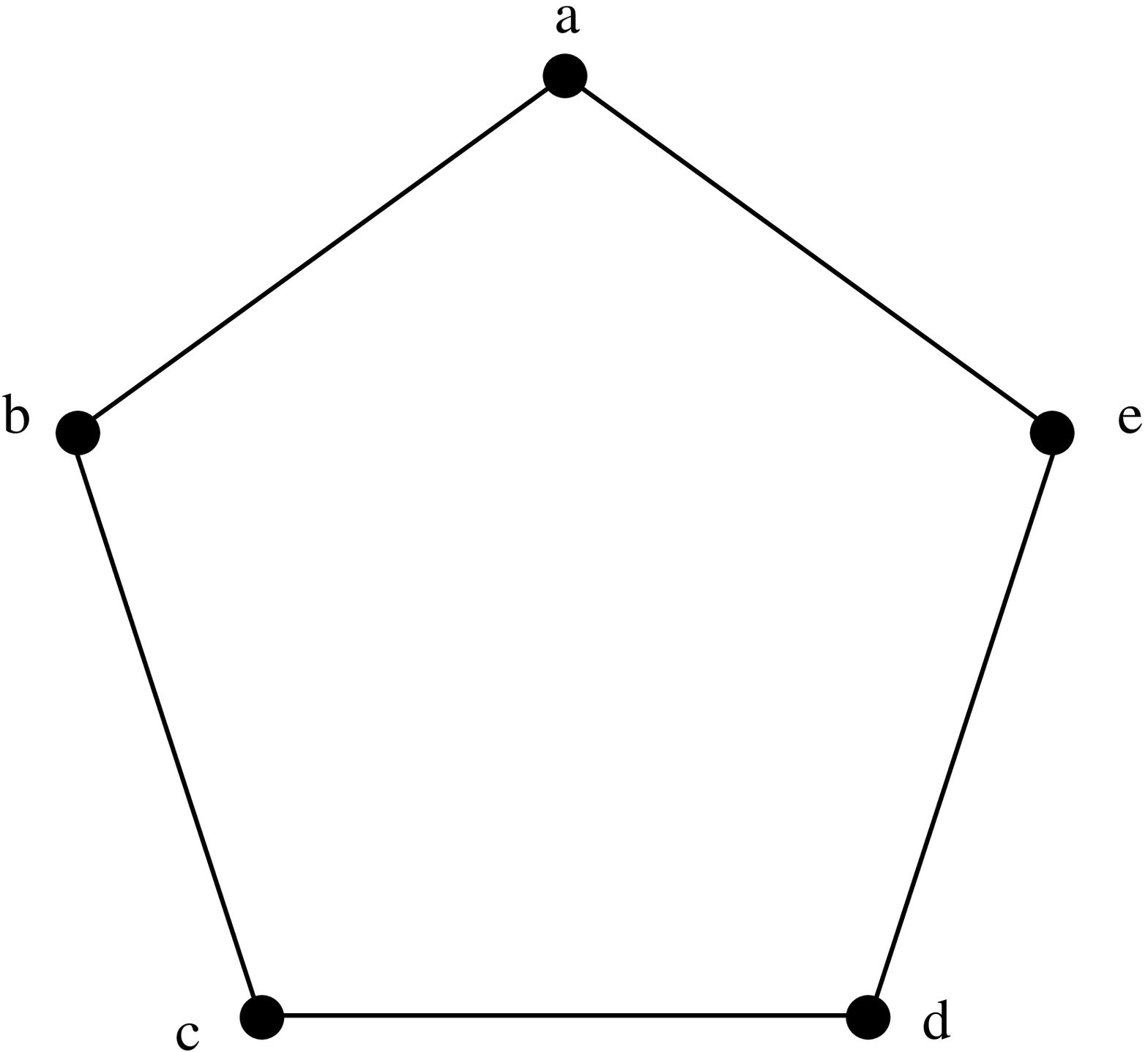}
\caption{$\Gamma_1=C_5$}
\label{fig:3a}
\end{center}
\end{subfigure}
\begin{subfigure}[b]{.45\columnwidth}
\begin{center}
\includegraphics[scale=0.25]{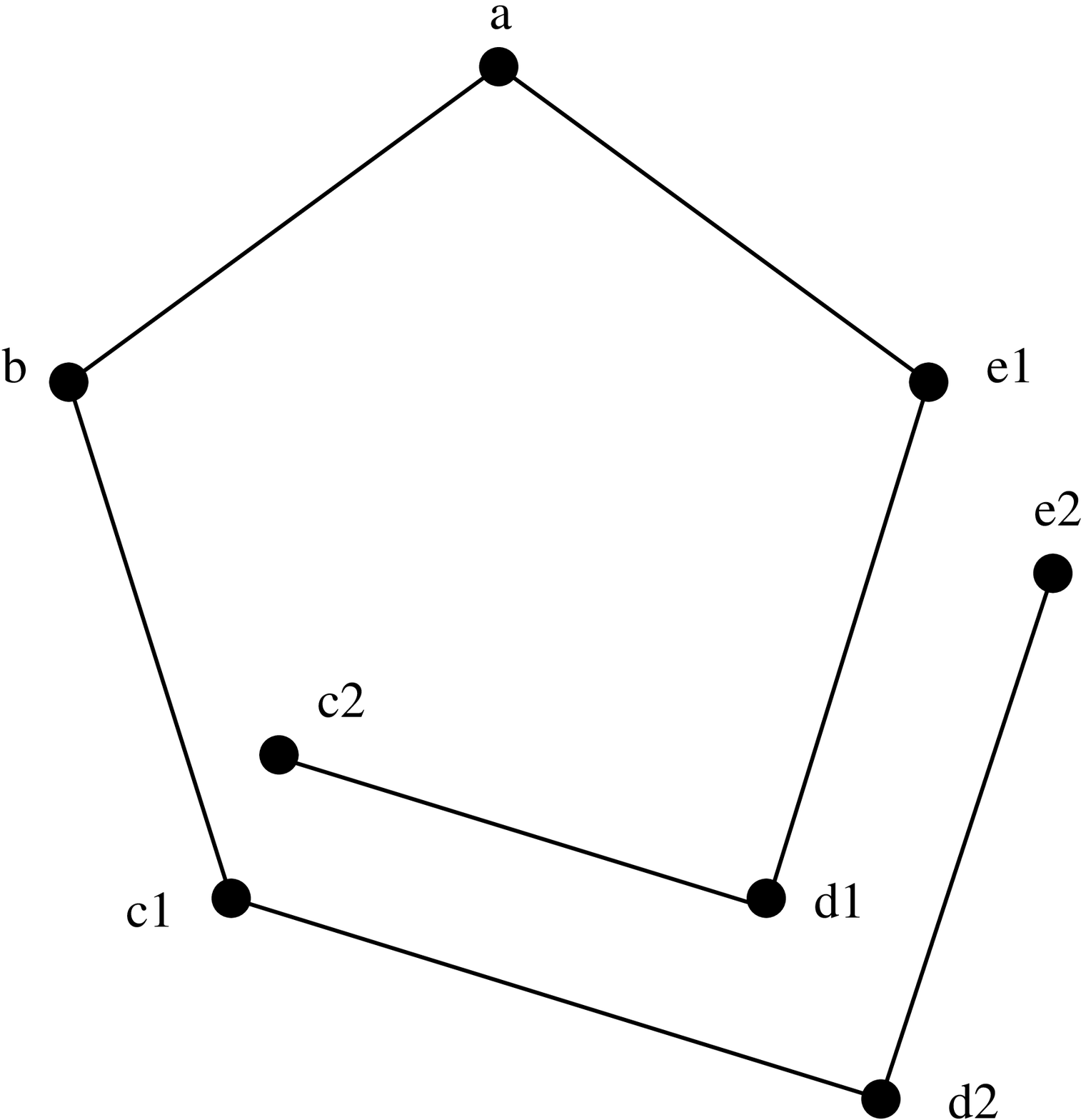}
\end{center}
\caption{$\Gamma_2 =P_7$}
\label{fig:3b}
\end{subfigure}
\end{center}
\caption{}\label{fig:3}
\end{figure}

\begin{rem} \label{rem:cp}
Observe that any for $n\ge 0$ the path graph $P_n$, and thus also its
complement $\overline P_n$, are weakly chordal. 
Indeed, $P_n$ is $C_m$ free, for all $m,n\ge 1$, and for $n \ge 3$, the 
graph $\overline P_n$ has $2$ vertices of degree $n-1$ and 
$n-1$ vertices of degree $n-2$. Therefore, for $n\ge 5$,  $\overline P_n$ 
contains no induced
$C_m$; and so its complement  $P_n$ contains no $\overline C_m$, $m\ge 1$. 
Hence $P_n$ is weakly chordal, for all integers
$n\ge 0$.
Furthermore, the complement of $C_5$ is again $C_5$.
 Moreover, $\GG(\overline P_n)\le\GG(\overline P_m)$, 
as $P_n\le P_m$, for all $n \le m$, 
and $\GG(C_n)\le \GG(C_5)$, for all $n\ge 5$, see \cite[Theorem 11]{KK}. 

\end{rem}

Define a map
\[
\varphi =\left\{ 
\begin{array}{l}
a\mapsto a;\\
b\mapsto b;\\
c\mapsto c_1c_2;\\
d\mapsto d_1d_2;\\
e\mapsto e_1e_2.
\end{array}
\right.
\]

\begin{prop}\label{prop:57}
The map $\varphi$ defined above induces an embedding of $\GG(\overline C_5)$ into $\GG(\overline P_7)$. 
\end{prop}
The following corollary follows immediately from Remark \ref{rem:cp} and 
Proposition \ref{prop:57}. 
\begin{cor}
For all $n\ge 5$ and $m\ge 7$, the group 
$\GG(C_n)$ embeds into $\GG(C_5)=\GG(\overline C_5)$,  
the group $\GG(\overline P_7)$ embeds into  $\GG(\overline P_m)$ 
and the group $\GG(C_n)$ embeds into $\GG(\overline P_m)$.
\end{cor}

Proposition \ref{prop:57} may also be used to 
show that  many surface groups embed into
$\GG(\overline P_m)$, for $m\ge 7$.  
\begin{cor}\label{cor:surfaces}
The fundamental group of a compact  surface of even Euler characteristic
at most $-2$ embeds into $\GG(C_5)\le \GG(\overline P_m)$, for all $m\ge 7$. 
\end{cor}
\begin{proof}
The compact orientable and non-orientable surfaces 
of Euler characteristic $-2$ are shown by Crisp and Wiest
\cite{CW} to embed in $\GG(C_5)$. As observed in \cite{R} the orientable
surface of Euler characteristic $-2$ is finitely covered by 
all  orientable surfaces
of smaller Euler characteristic; so the latter all have fundamental groups
which embed in $\GG(C_5)$. The construction of these finite covers appears 
in \cite[Example 2.6]{M} and a similar construction gives a finite
 cover of the compact  non-orientable surface of Euler characteristic $-2$ by 
any compact non-orientable surface of even Euler characteristic, less than
$-2$.   
\end{proof}

(Using Stallings' definition of the genus of a closed surface $S$,
namely $\genus(S)=(2-\chi(S))/2$, the Corollary says that all surface groups
of integer genus $g$, such that  $g\ge 2$, embed in 
$\GG(\overline P_m)$, $m\ge 7$.)

\begin{proof}[Proof of Proposition \ref{prop:57}]
As $\varphi$ maps each of the relators of the presentation \eqref{eq:c5} of 
$\GG(\overline C_5)$ to the identity of $\GG(\overline P_7)$, it is immediate that $\varphi$ is a homomorphism. We show that $\varphi$ is injective. 

Let $w$ be a reduced, non-trivial, word in $\GG(\overline C_5)$, 
suppose that $\varphi(w)=1$, and assume that $w$ is of minimal length,
 among all such words.  Let $\mathcal{D}$ be a minimal van Kampen 
diagram for $\varphi(w)$. As in the proof of Lemma \ref{lem:1}, 
$\varphi(w)$ must contain a letter from 
$\{c_i^{\pm 1}, d_i^{\pm 1}, e_i^{\pm 1}\}$. 
Suppose first that $\varphi(w)$ contains an occurrence of $e_i^{\pm 1}$. 
Then  $\varphi(w)$ contains $e_1^{\pm 1}$ and, 
as in the proof of Lemma \ref{lem:1},
without loss of generality we may assume there exists a band 
$L_v$ such that  $v=e_1$
 and $w(L_v)=e_2w_2e_2^{-1}$, where $w_2$ contains no occurrences 
of $e_i^{\pm 1}$, $i=1,2$. 
As $w$ is reduced in $\GG(\G_1)$ it follows that $w(L_v)$ must contain a letter from
$\{a^{\pm 1}, d_i^{\pm 1}\}$. 

In the case where $w(L_v)$ contains no letter $d_i^{\pm 1}$, the word 
$w_2$ contains an occurrence of $a^{\pm 1}$ and,  as $[a,e_1]\neq 1$,  
 no band $L_a$ can cross $L_v$, so a  band 
with one end at an occurrence of $a^{\pm 1}$ 
in $w_2$ must have both ends in $w_2$. Hence there must 
be a band $L_x$ such that $x=a^{\pm 1}$, $w(L_x)$ is a subword of $w_2$ and 
no letter $a^{\pm 1}$ occurs in $w(L_x)$. As $w$ is reduced,  
$\az(w(L_x))$ must contain a letter from $\{b,e_i\}$; but $w(L_x)$ is
 a subword of $w(L_v)$, so this implies $b^{\pm 1}$ must occur in $w(L_x)$. 
Now we find a band $L_y$, where $y=b^{\pm 1}$, $w(L_y)$ is a subword 
of $w(L_x)$ and contains no occurrence of $b^{\pm 1}$. We are forced
to conclude that $w(L_y)\in \langle c_1,c_2\rangle$. As $w$ is 
reduced in $\GG(\G_1)$ the word $w(L_y)$ contains 
an occurrence of $c_i^{\pm 1}$  and, as all such letters occur in
subwords $(c_1c_2)^{\pm 1}$, it follows that $w(L_y)$ contains an 
occurrence of $c_2^{\pm 1}$. As $[b,c_2]\neq 1$,  there is a band $L_z$, where 
$z=c_2^{\pm 1}$, with both
ends in $w(L_y)$.  In this case $w(L_y)$ contains the ends $c_2^\epsilon$ and 
$c_2^{-\epsilon}$ of $L_z$, and, as  $w(L_y)\in \langle c_1,c_2\rangle$, 
 $w$ cannot be reduced, a contradiction

Hence we assume that $w(L_v)$ contains an occurrence $d_i^{\pm 1}$. As 
$[d_1,e_1]\neq 1$,  
 we may choose a band $L_x$, such that such that $x=d_1^{\pm 1}$, and  
$w(L_x)$ is a subword of $w(L_v)$ which either contains no 
letter $d_i^{\pm 1}$, if $x=d_1^{-1}$, or factors as 
$w(L_x)=d_2w_3d_2^{-1}$, 
where  no letter $d_i^{\pm 1}$ occurs in
$w_3$, if $x=d_1$. 
Next, we may choose a band $L_y$, such that $y=c_1^{\pm 1}$, 
$w(L_y)$ is a subword of $w(L_x)$ and contains no letter $c_i^{\pm 1}$,
except possibly the first and last (which may be $c_2$ and $c_2^{-1}$). 
Finally there must be a band $L_z$, with $z=b^{\pm 1}$, such that 
$w(L_z)$ is a
subword of $w(L_y)$ and there is no occurrence of $b^{\pm 1}$ in $w(L_z)$. 
This forces $w(L_z)$ to be an element of $\langle a \rangle$, which 
cannot be trivial as the ends of $L_z$ must be separated by at least one
letter $a^{\pm 1}$, if  $w$ is reduced.  However, there is then a band
with ends $a^{\epsilon}$ and $a^{-\epsilon}$, both lying in $w(L_z)$; and so $w$ is not reduced, a contradiction. 

If $e_i\notin \az(w)$, $i=1,2$, then begin  above with 
$d_i$ instead of $e_i$ and the same  argument gives the 
result. The case where both $e_i, d_i\notin \az(w)$ follows similarly.  
\end{proof}
\section{Subgroups of $C_4$ and $P_3$ free graphs}\label{sec:Degc}
Let us call 
graphs which are $C_4$ and $P_3$ free \emph{thin-chordal graphs}. 
The following theorem is a generalisation of  
part of the result of Droms \cite[Theorem]{D}. 

\begin{thm}\label{thm:tcegc}
The extension graph conjecture holds for finite, thin-chordal graphs, 
in the following strong form. Let $\G$ be a finite thin-chordal graph.
Then $H$ is a subgroup of $\GG(\G)$ if and only if $H\cong \GG(\G')$, 
 for some induced subgraph $\G'$ of $\G^e$.   
\end{thm}
\begin{proof}
Let  $\G$ be a  thin-chordal graph. If $H\cong \GG(\G')$, 
 for an induced subgraph $\G'$ of $\G^e$, then it follows   
from \cite[Theorem 2]{KK} that $H$ is a subgroup 
of $\GG(\G)$. To prove the converse, we use induction on the number of vertices of $\G$.  
If $\G$ is not connected then $\GG(\G)$ is a free
product of partially commutative groups of thin-chordal graphs, each
of which has fewer vertices than $\G$. Hence the result holds for 
each factor, by induction. The result for $\GG(\G)$ follows from the 
Kurosh subgroup theorem. 

To see this, assume that the connected components
of $\G$ are $\G_1,\ldots, \G_n$ and let $\GG_i=\GG(\G_i)$. From
\cite[Lemma 26 (2)]{KK}, it follows that $\G^e$ is the disjoint union
of countably many copies of $\G_i^e$, $i=1,\ldots ,n$. That
is 
\[
\G^e=\coprod_{i=1}^n\left(\coprod_{j\in\NN}\G_{i,j}^e\right),
\]
where $\G_{i,j}^e=\G_i^e$, for all $j\in \NN$ and $i=1,\ldots, n$. 
Let $D_i$ be a set of representatives of double cosets $Hx\GG_i$ 
of $H$ and $\GG_i$ in $\GG(\G)$. Then, from the Kurosh subgroup theorem,
\[H\cong \FF\ast \ast_{i=1}^n\left(\ast_{d\in D_i}H\cap d\GG_id^{-1}\right),
\]
where $\FF$ is a free group of finite or countably infinite rank.

By induction, we have 
\[H\cap d\GG_id^{-1}\cong d^{-1}Hd\cap \GG_i\cong \GG(\Lambda_{i,d}),\]
where $\L_{i,d}\le \G_i^e$, for all $d\in D_i$, $i=1,\ldots, n$. 
Let $X_0$ be a free generating set for $\FF$. Then, for each $i$, there
exists a bijection $\a_i:X_0\coprod D_i\maps \NN$. Hence we 
have an embedding of graphs, 
\[X_0\coprod \coprod_{d\in D_i}\L_{i,d}\le \coprod_{j\in \NN}\G_{i,j}^e,\]
where $x\in X_0$ maps to a vertex of $\G_{i,\a_i(x)}$ and 
$\L_{i,d}$ is embedded in $\G_{i,\a(d)}^e=\G_i^e$, for all $d\in D_i$. 
Hence 
\[\FF\ast\left(\ast_{d\in D_i} H\cap d\GG_id^{-1}\right)
\cong \GG(X_0)\ast\left(\ast_{d\in D_i} \GG(\L_{i,d})\right)
\cong \GG\left(X_0\coprod\coprod_{d\in D_i}\L_{i,d}\right),
\]
where 
\[
X_0\coprod\coprod_{d\in D_i}\L_{i,d}
\le \coprod_{j\in \NN}\G_{i,j}^e.
\]
Therefore 
\[H\cong \GG(X_0)\ast\left(\ast_{i=1}^n\left(\ast_{d\in D_i} \GG(\L_{i,d})\right)\right)
\cong \GG\left(X_0\coprod\left(\coprod_{i=1}^n\coprod_{d\in D_i}\L_{i,d}\right)\right),
\]
where 
\[
X_0\coprod\left(\coprod_{i=1}^n\coprod_{d\in D_i}\L_{i,d}\right)
\le\coprod_{i=1}^n \left(X_0\coprod\coprod_{d\in D_i}\L_{i,d}\right)
\le \coprod_{i=1}^n\coprod_{j\in \NN}\G_{i,j}^e= \G^e.
\]

Suppose then that $\G$ is connected and that 
the result holds for all graphs of fewer vertices, connected or not.  
Let $H$ be a subgroup of $\GG(\G)$. Then, as shown in   \cite[Lemma]{D},
$\G$ has a vertex $z$ which is connected to every other vertex of $\G$. 
Let $V(\G)=Y\cup \{z\}$ and let $\L$ be the induced subgraph of $\G$ with
vertices $Y$.  Then $\L$ is a thin-chordal graph and, from 
\cite[Lemma 26 (1)]{KK}, $\G^e=\{z\}\ast\L^e$, that is, the graph formed from
$\L^e$ by adjoining a new vertex $z$ connected to every other vertex. 
It is shown in \cite{D} that either $H\le \GG(\L)$ or 
$H\cong \langle z^d\rangle \times H'$, where $H'\le \GG(\L)$, and $d$ is
 a positive integer. By induction on the number of vertices, $H$ in the 
first case, or $H'$ in the second case,  is 
a partially commutative group 
isomorphic to $\GG(\L')$, for some induced  subgraph $\L'$ of $\L^e$. 
As $\L^e$ is a full subgraph of $\G^e$ this completes the proof in the 
case $H\le \GG(\L)$. In the second case $H\cong \langle z^d\rangle \times 
\G(\L')\cong \GG(\{z\}*\L')$ and $\{z\}*\L'\le \{z\}*\L^e=\G^e$, so the result holds in this
case as  well.  
\end{proof}

\end{document}